\documentclass[12pt]{amsart}
\usepackage{amssymb,latexsym,amsmath,enumerate,nicefrac}

\usepackage{amsaddr}
\usepackage[total={6.5in,9in}, top=1in, left=1in, headsep = 0in]{geometry}

\usepackage{fancyhdr}
\usepackage{color}

\definecolor{egyptianblue}{rgb}{0.06, 0.2, 0.65}
\definecolor{cobalt}{rgb}{0.0, 0.28, 0.67}
\definecolor{yaleblue}{rgb}{0.06, 0.3, 0.57}
\definecolor{maroon}{rgb}{0.5, 0.0, 0.0}
\definecolor{indigo}{rgb}{0.0, 0.25, 0.42}
\definecolor{armygreen}{rgb}{0.29, 0.33, 0.13}
\definecolor{huntergreen}{rgb}{0.21, 0.37, 0.23}
\definecolor{royalblue}{rgb}{0.0, 0.14, 0.4}
\definecolor{zaffre}{rgb}{0.0, 0.08, 0.66}

\definecolor{myurlcolor}{rgb}{0.5, 0.0, 0.0}
\definecolor{mycitecolor}{rgb}{0.06, 0.2, 0.65}
\definecolor{myrefcolor}{rgb}{0.06, 0.2, 0.65}

\usepackage[bookmarks=false,pagebackref=true]{hyperref}
\hypersetup{colorlinks,
linkcolor=myrefcolor,
citecolor=mycitecolor,
urlcolor=myurlcolor}

\usepackage[all]{xy}
\usepackage{tikz}
\usetikzlibrary{calc}
\usepackage{graphics}
\usepackage{dsfont}
\usepackage{amsfonts}
\usepackage{mathtools}
\usepackage{comment}

\DeclareMathOperator{\Vol}{Vol}

\renewcommand{\d}{{\rm d}}

\pgfkeys{/tikz/.cd, arrowhead/.default=-1pt, arrowhead/.code={}}

\allowdisplaybreaks

\usepackage{graphics}
\usepackage{dsfont}
\usepackage{amsbsy}
\usepackage{amsmath}
\usepackage{amsfonts,amssymb}
\usepackage{mathrsfs}
\usepackage{mathtools}

\usepackage{tikz}
\usepackage{bm}
\usetikzlibrary{calc}

\newcommand{\x}{[x]}

\newcommand{\Prob}{{\rm Prob}}
\newcommand{\Pb}{{\rm P}}
\newcommand{\Pbm}{{\rm P}^m}

\newcommand{\EXm}{{\mathds E}_m}

\newcommand{\Gm}{G_m}
\newcommand{\Gmdual}{\tilde{G}_m}
\newcommand{\BB}{{\mathds B}}
\renewcommand{\SS}{{\mathds S}}
\newcommand{\Sm}{{\mathds S}_m}
\newcommand{\Bm}{{\mathds B}_m}
\newcommand{\Bmdual}{\tilde{\mathds B}_m}
\newcommand{\Smdual}{\tilde{\mathds S}_m}

\newcommand{\Zp}{\mathds{Z}_p}
\newcommand{\Qp}{\mathds{Q}_p}

\newcommand{\M}{\mathcal{M}}

\newcommand{\e}{{\rm e}}
\newcommand{\X}{\tilde{X}}
\newcommand{\Y}{\tilde{Y}}

\newcommand{\F}{\mathcal F}
\newcommand{\Fm}{\mathcal F_m}

\newcommand{\C}{\mathscr C}

\renewcommand{\)}{\right)}
\DeclarePairedDelimiter\floor{\lfloor}{\rfloor}

\DeclarePairedDelimiter\Gmap{[}{]}
\DeclarePairedDelimiter\Gdmap{\langle}{\rangle}

\newcommand{\y}{\Gdmap{y}}
\renewcommand{\x}{\Gmap{x}}

\DeclarePairedDelimiter\ceil{\lceil}{\rceil}

\newtheorem{theorem}{Theorem}[section]
\newtheorem{lemma}{Lemma}[section]
\newtheorem{proposition}{Proposition}[section]
\theoremstyle{definition}

\theoremstyle{remark}

\makeatletter
\def\namedlabel#1#2{\begingroup
    #2%
    \def\@currentlabel{#2}%
    \phantomsection\label{#1}\endgroup
}
\makeatother

\begin{document}

\title{Brownian Motion in the $\bm p$\,-Adic Integers is a Limit of Discrete Time Random Walks}
\author{Tyler Pierce$^\ast$ and David Weisbart$^\ast$}
\address{\begin{tabular}[h]{cc}
 $^\ast$Department of Mathematics\\
   University of California, Riverside \\&\end{tabular}}
\email{tpier002@ucr.edu}\email{weisbart@math.ucr.edu}

\begin{abstract}
Vladimirov defined an operator on balls in $\Qp$, the $p$-adic numbers, that is analogous to the Laplace operator in the real setting.  Kochubei later provided a probabilistic interpretation of the operator.  This \emph{Vladimirov-Kochubei operator} generates a real-time diffusion process in the ring of $p$-adic integers, a Brownian motion in $\Zp$.  The current work shows that this process is a limit of discrete time random walks.  It motivates the construction of the Vladimirov-Kochubei operator, provides further intuition about the properties of ultrametric diffusion, and gives an example of the weak convergence of stochastic processes in a profinite group.
\end{abstract}

\maketitle

\tableofcontents

%


\thispagestyle{empty}


\section{Introduction}


For any prime $p$ and any positive real number $b$, the Vladimirov operator $\Delta_b$ (see Section~\ref{Sec:TheLimitingProcess}) is an example of one type of pseudodifferential operator on the field of $p$-adic numbers, $\Qp$.  The Vladimirov operator is the $p$-adic analog of the fractional Laplacian in the real setting.  Take $\rho$ to be the fundamental solution to the parabolic pseudodifferential equation \begin{equation}\label{Intro:HeatEquation}\frac{{\rm d}}{{\rm d}t}\rho(t,x) = -\Delta_b\rho(t,x).\end{equation} The indexed collection of functions $(\rho(t, \cdot))_{t>0}$ forms a convolution semigroup of probability density functions that gives rise to a probability measure $\Pb$ on $D([0, \infty)\colon \Qp)$, the Skorohod space of $\Qp$-valued c\`{a}dl\`{a}g functions on $[0,\infty)$ endowed with the Skorohod metric \cite{Varadarajan:LMP:1997}.  The measure $\Pb$ is a $p$-adic analog of the Wiener measure on the space of continuous, real-valued paths.  Take $Y$ to be the function on $[0, \infty)\times D([0, \infty)\colon \Qp)$ that is defined on any pair $(t, \omega)$ by \begin{equation}\label{Intro:Eq:YDef}Y(t,\omega) = \omega(t).\end{equation} Curry variables to obtain for each $t$ in $[0,\infty)$ the random variable $Y_t$ on $D([0, \infty)\colon \Qp)$ that acts on any $\omega$ in $D([0, \infty)\colon \Qp)$ by \begin{equation}\label{Intro:Eq:YtDef}Y_t(\omega) = \omega(t).\end{equation} Because of the close formal analogy between the Vladimirov operator and the Laplace operator, the current work refers to the stochastic process $(D([0, \infty)\colon \Qp), \Pb, Y)$ as a $p$-adic Brownian motion.

Authors traditionally refer to diffusion processes of the type that the current paper studies as ultrametric diffusion processes \cite{ABKO:JPA:2002, Bik:UAA:2010}, but such processes may potentially be more general.  In some respects, a true analog of real Brownian motion in the $p$-adic setting is the process that Bikulov and Volovich introduced \cite{BV:IzvMath:1997}.  This is a very different stochastic process, one that is parameterized by a $p$-adic time variable.

Motivated by the recent works of Khrennikov, Oleschko, and Correa Lopez on $p$-adic models of a porous medium \cite{Khrennikov:JFAA:2016, Khrennikov:Entropy:2016}, Kochubei recently studied both linear and nonlinear heat equations on balls in $\Qp$ \cite{Koch:UMJ:2018}.  These works motivate the current investigation.  There is a challenge even in the real setting, as Kochubei pointed out, of interpreting the meaning of a non-local operator on a bounded domain \cite{BonoforteVazquez:NLA:2016}.  Vladimirov defined a version of the operator $\Delta_{b,N}$ on the ball $\BB(N)$, the ball of radius $p^N$ in $\Qp$ centered at $0$ \cite{Vlad90}.  Kochubei gave a probabilistic interpretation of the operator $\Delta_{b,N}$ \cite{Kochubei:Book:2001}, and explicitly gave the heat kernel for this operator \cite[Theorem~4.9]{Kochubei:Book:2001}.

The study of Brownian motion in non-Archimedean settings is not only of intrinsic mathematical interest but also appears to have substantial physical applications. For example, Parisi's early work on spin glasses \cite{Parisi:PRL:1979, Parisi:JPA:1980} implicitly involved ultrametric structures.  Later, Parisi and Sourlas \cite{Parisi_Sourlas:JPA:1999} and Avetisov, Bikulov, and Kozyrev \cite{Avetisov_Bikulov_Kozyrev:JPA:1999} independently proposed $p$-adic models to describe replica symmetry breaking. Significant applications of $p$-adic diffusion in physics which have received experimental confirmation are found in the works of Avetisov, Bikulov, and Zubarev \cite{ABZ:ProcSteklov:2014}, and Bikulov and Zubarev \cite{BZ:Physica:2021}, who utilized the Vladimirov equation to describe protein molecule dynamics.  Since Brownian motion in the real setting is a scaling limit, a general question arises:  In what sense can sequences of discrete time random walks approximate Brownian motion in non-Archimedean settings?  

Earlier works demonstrate that Brownian motion in several non-Archimedean settings are scaling limits \cite{BW, WJPA, W:Expo:24}. The first of these works \cite{BW} shows that $p$-adic Brownian motion is a limit of discrete-time random walks. The second work \cite{WJPA} simplified the approach and improved the results of the first by showing that $p$-adic Brownian motion is a scaling limit of a single discrete-time process valued in a discrete group. This extended the results to a wider class of Vladimirov operators and demonstrated the weak convergence of the processes on unbounded time intervals, rather than just compact time intervals. The latest work \cite{W:Expo:24} extended the second to the setting of Brownian motion in any finite-dimensional vector space over any non-Archimedean local field of any characteristic. The current work provides what appears to be the first example of such results in the setting of a profinite group by demonstrating that any Brownian motion with paths in the ring of $p$-adic integers is a limit of discrete-time random walks.

Section~\ref{Sec:TheLimitingProcess} reviews the basic theory of harmonic analysis on $\Qp$ and its unit ball $\Zp$, and presents in a restricted setting the operator that Vladimirov and Kochubei developed and studied, the \emph{Vladimirov-Kochubei operator} on $\Zp$.  It also reviews some properties of the Brownian motion associated with the Vladimirov-Kochubei operator.  Section~\ref{Sec:TheDiscreteSettings} develops for each natural number $m$ a discrete time random walk on $\Zp$ by first constructing a random walk on each finite group $\Zp\slash p^m\Zp$, and then by embedding the finite groups into $\Zp$ to obtain a sequence of embedded processes.  Section~\ref{Sec:Moments} establishes moment estimates for the processes in the finite groups.  Given an appropriate time scale for each embedded process, Section~\ref{Sec:Convergence} establishes a uniform moment estimate for the sequence of measures $(\Pbm)$ on $D([0, \infty)\colon \Zp)$ that are associated with the embedded processes.  These moment estimates show that the sequence of measures is uniformly tight \cite{cent}.  Section~\ref{Sec:Convergence} also establishes the convergence of the finite dimensional distributions of the approximating measures to the finite dimensional distributions of the measure $\Pb$ for the process that Kochubei constructed.  The convergence of the finite dimensional distributions together with the uniform tightness of the sequence of approximating measures implies Theorem~\ref{Sec:Con:Theorem:MAIN}, the main result of the current work.  

\begin{theorem}\label{Sec:Con:Theorem:MAIN}
The sequence of measures $(\Pbm)$ converges weakly to $\Pb$.
\end{theorem}

It is in this precise sense that Brownian motion in $\Zp$ is a limit of discrete time random walks.


\section{The limiting processes}\label{Sec:TheLimitingProcess}



\subsection{Path Spaces}

The current work involves the construction and convergence of probability measures on a space of paths $\Omega(I)$ with state space $\mathcal S$ equal to either a finite Abelian group, $\Qp$, or a ball in $\Qp$, and domain $I$ equal to $[0, \infty)$ or $\mathds N_0$, the natural numbers with $0$.  There is a general procedure for constructing such measures.  In the case that $I$ is the interval $[0,\infty)$ and $\mathcal S$ is $\Qp$ or a ball in $\Qp$, the fundamental solution $\rho$ to an analogue of the heat equation in $\mathcal S$ gives for each $t$ in $[0, \infty)$ a probability density function $\rho(t, \cdot)$.  The indexed collection $(\rho(t, \cdot))_{t>0}$ is a convolution semigroup of probability density functions.

Any finite sequence $h$ of pairs of the form \[h = ((t_0, U_0), (t_1, U_1), \dots, (t_n, U_n)),\] where the sequence $(t_0, t_1, \dots, t_n)$ is a strictly increasing finite sequence in $[0, \infty)$, $t_0$ is equal to $0$, and each $U_i$ is a Borel subset of $\mathcal S$, is a \emph{history} for paths in $\mathcal S$.  The sequence $(0, t_1, \dots, t_n)$ is the \emph{epoch} for $h$ and the sequence $(U_i)_{i\in\{0, 1, \dots, n\}}$ is the \emph{route} for $h$. Take $\ell(h)$ to be $n$, the \emph{length of the history}.  For each $i$ in $\{0, 1, \dots, n\}$, associate an epoch $e_h$ and a route $U_h$ with a given history $h$ by writing \[e_h(i) = t_i \quad \text{and} \quad U_h(i) = U_i.\]  

Take $H(I, \mathcal S)$ to be the set of all histories for $\mathcal S$-valued paths with domain $I$.  On specifying a path space $\Omega(I)$, take $\C$ to be defined for any $h$ in $H(I, \mathcal S)$ by \begin{align*}\C(h) &= \left\{\omega \in \Omega(I)\colon \forall i\in \mathds N_0\cap [0, \ell(h)], \; \omega(e_h(i)) \in U_h(i)\right\}.\end{align*}  For any history $h$, $\C(h)$ is a \emph{simple cylinder set}.  The set $\C(H(I, \mathcal S))$ is the \emph{set of all simple cylinder sets of $\Omega(I)$} and is a $\pi$-system, but not an algebra. The \emph{set of cylinder sets} is the $\sigma$-algebra generated by $\C(H(I, \mathcal S))$---its elements are the \emph{cylinder sets}.  

The Kolmogorov extension theorem guarantees that there is a measure $\Pb$ on the cylinder sets of $F([0, \infty)\colon \mathcal S)$, the space of all paths from $[0, \infty)$ to $\mathcal S$, so that for any history $h$, if $U_h(0)$ contains $0$, then%
\begin{align}\label{Framework:EQ:FormulaFDDensity}
\Pb(\C(h)) = \int_{U_h(1)}\cdots \int_{U_h(\ell(h))} \prod_{i=1}^{\ell(h)}\rho(e_h(i)-e_h(i-1), x_i-x_{i-1})\,{\rm d}\mu(x_1)\cdots{\rm d}\mu(x_{\ell(h)}),
\end{align}
 and if $U_h(0)$ does not contain $0$, then $\Pb(\C(h))$ is equal to $0$.  In all current settings, certain moment estimates for $\Pb$ will guarantee that there is a version of the stochastic process $(F([0, \infty)\colon \mathcal S), \Pb, Y)$ with sample paths in $D([0, \infty)\colon \mathcal S)$ \cite{bil1, cent}.  In the case when $I$ is $\mathds N_0$ and $\mathcal S$ is a discrete space, the measure $\mu$ will be a counting measure, or weighted counting measure, and the integrals will be sums.

The probabilities that are given by \eqref{Framework:EQ:FormulaFDDensity} are the \emph{finite dimensional distributions} for the stochastic process $(D([0, \infty)\colon \mathcal S), \Pb, Y)$.  It is frequently useful to specify the finite dimensional distributions for a stochastic process by using the properties of stochastic processes before knowing that there is, indeed, a stochastic process with the given finite dimensional distributions.  Refer to the construction of the finite dimensional distributions for a stochastic process as the construction of the \emph{abstract stochastic process} $\tilde{Y}$ and the construction of the law for a random variable without actually specifying its domain as the construction of an \emph{abstract random variable} $\tilde{X}$.  A model for the abstract random variable $\tilde{X}$ is a random variable with the same law as $\tilde{X}$.  A model for the abstract stochastic process $\tilde{Y}$ is a stochastic process with the same finite dimensional distributions as $\tilde{Y}$.  Distinguish probabilities associated with abstract random variables and stochastic processes from those associated with their models by writing ${\rm Prob}(\tilde{X}\in A)$ for the probability that $\tilde{X}$ takes a value in some Borel set $A$.  Use an analogous notation for $\tilde{Y}$.


\subsection{Brownian motion in $\mathds Q_p$}

The field of $p$-adic numbers $(\Qp, +, \times)$ is the completion of the rational numbers with respect to the $p$-adic absolute value, $|\cdot|$.  For each integer $k$, denote by $\BB(k)$ and $\SS(k)$ the sets \[\BB(k) = \{x\in\mathds Q_p\colon |x|_p \leq p^k\} \quad \text{and}\quad \SS(k) = \{x\in\mathds Q_p\colon |x|_p = p^k\},\] and by $\mathds Z_{p}$ the \emph{ring of integers}, the set $\BB(0)$.  Take $\mu$ to be the unique Haar measure on the locally compact Abelian group $(\mathds Q_p, +)$ that gives $\mathds Z_p$ unit measure.

For any $x$ in $\mathds Q_{p}$, there is a unique function $a_x$ with \begin{equation}\label{Def:ax}a_x\colon\mathds Z\to\{0, 1, 2, \dots, p-1\}\quad \text{and}\quad  x = \sum_{k\in\mathds Z}a_x(k)p^k.\end{equation}  Denote by $\{x\}$ the \emph{fractional part of $x$}, where \[\{x\} = \sum_{k<0} a_{x}(k)p^k,\] and by $\chi$ the additive rank 0 character on $\Qp$ that is given by \[\chi(x) = e^{2\pi{\sqrt{-1}}\{x\}}.\]  

The Pontryagin dual of $\Qp$ is again $\Qp$, and every character on $\Qp$ can be written as $\chi(z\,\cdot\,)$ for some $z$ in $\Qp$, where for any $y$ in $\Qp$, \[\chi(z\,\cdot\,)(y) = \chi(zy).\] The Fourier transform $\mathcal F$ and inverse Fourier transform $\mathcal F^{-1}$ are unitary operators on $L^2(\mathds Q_{p})$ that are given for any $f$ in $L^2(\mathds Q_p)\cap L^1(\mathds Q_p)$ by \[(\mathcal Ff)(x) = \int_{\mathds Q_{p}}\chi(-xy)f(y)\,{\rm d}\mu\!\(y\) \quad\text{and}\quad (\mathcal F^{-1}f)(y) = \int_{\mathds Q_{p}}\chi(xy)f(x)\,{\rm d}\mu(x).\]

Denote by $SB(\Qp)$ the Schwartz-Bruhat space of compactly supported, locally constant, $\mathds C$-valued functions.  Take ${\mathcal M}$ to be the multiplication operator that acts on $SB(\Qp)$ by \[({\mathcal M}f)(x) = |x|^bf(x).\] Denote by $\Delta^\prime_b$ the unique self-adjoint extension of the essentially self-adjoint operator that acts on any $f$ in $SB(\Qp)$ by \[(\Delta^\prime_b f)(x) = \big(\F^{-1}{\mathcal M}\F f\big)\!(x).\]  

For any $\mathds C$-valued function $g$ with domain $(0, \infty)\times \Qp$ and any $t$ in $(0, \infty)$, take $g_t$ to be the function that is defined for every $x$ in $\mathds Q_p$ by \[g_t(x) = g(t,x).\]  Denote by $\mathcal D(\Delta_b)$ the set of all such $g$ so that for all $t$ in $(0, \infty)$, $g_t$ is in the domain of $\Delta^\prime_b$.  Take $\Delta_b$ to be the \emph{Vladimirov operator with exponent $b$} that acts on any $g$ in $\mathcal D(\Delta_b)$ by \[(\Delta_b g)(t,x) =  (\Delta_b g_t)(x).\]  Similarly extend the Fourier transform to act on $\mathds C$-valued functions on $(0, \infty)\times \Qp$ that are square integrable for each positive $t$ and the operator $\frac{{\rm d}}{{\rm d}t}$ to act on $\mathds C$-valued functions on $(0, \infty)\times \Qp$ that for any $x$ in $\mathds Q_p$ are differentiable in the $t$ variable.

Take $D$, the diffusion coefficient, to be any positive real number.  Denote by $\rho$ the fundamental solution to \eqref{Intro:HeatEquation}, which for any $(t,x)$ in $(0, \infty)\times\mathds Q_p$ is given by \begin{equation}\label{LimProc:Eq:QpDiff}\rho(t,x) = \left(\mathcal F^{-1}\e^{-D t|\cdot|^b}\right)\!(x).\end{equation}  For any positive $t$, the Fourier transform of $\rho(t, \cdot)$ is the function $\phi(t, \cdot)$ that is given by \begin{equation}\label{LimProc:Eq:QppdfChar}\phi(t,x) = \e^{-D t|x|^b},\end{equation} and so $\rho$ is given by 
\begin{equation}\label{LimProc:Eq:Qppdf}
\rho(t,x) = \sum_{k\in \mathds Z} p^k\left(\e^{-D tp^{kb}} - \e^{-D tp^{(k+1)b}}\right) \mathds 1_{\BB(-k)}(x).
\end{equation}  

Follow Varadarajan \cite{Varadarajan:LMP:1997} to see that there is a measure $\Pb$ on the space $D([0, \infty)\colon\Qp)$ that gives rise to the finite dimensional distributions given by \eqref{Framework:EQ:FormulaFDDensity}.  The triple $(D([0, \infty)\colon\Qp), \Pb, Y)$ is a $p$-adic Brownian motion.


\subsection{The $\mathds Z_p$ setting}

 Take $SB(\Zp)$ to be the locally constant functions on $\Zp$.  Define the function $\sharp$ from $SB(\Zp)$ to $SB(\Qp)$ by \[\sharp f(x) = \begin{cases}f(x) &\text{ if } x\in \Zp\\0 &\text{ otherwise.}\end{cases}\]  For any function $f$ that acts on $\Qp$ and any subset $X$ of $\Qp$, denote by $f\big|_{X}$ the restriction of $f$ to $X$.  Vladimirov and Kochubei both worked in the slightly more general setting of $p$-adic balls \cite{Vlad90}.  In the context of $\Zp$, define $\Delta^\prime_{b,0}$ to act on any $f$ in $SB(\Zp)$ by \[(\Delta^\prime_{b,0}f)(x) = (\Delta_b \circ \sharp f)\Big|_{\Zp}(x).\]  The \emph{Vladimirov-Kochubei operator} $\Delta_{b,0}$ is the self-adjoint closure of $\Delta^\prime_{b,0}$.  Kochubei gave a probabilistic interpretation of this operator and studied its properties \cite{Koch:UMJ:2018}.  The current section specializes to the ball $\Zp$, but includes a general, constant, diffusion coefficient $D$.
 
The Pontryagin dual of $\Zp$ is the discrete group $\Qp\slash\Zp$.  Take $\Gdmap{\cdot}$ to be the quotient map from $\Qp$ to $\Qp\slash\Zp$ that is given for each $x$ in $\Qp$ by \[\Gdmap{x} = x + \Zp.\]  Define the absolute value $|\cdot|$ on $\Qp\slash\Zp$ by \[|\Gdmap{x}| = \begin{cases}|x| &\mbox{if }\Gdmap{x}\ne \Gdmap{0}\\0 &\mbox{if }\Gdmap{x} = \Gdmap{0}.\end{cases}\] %
Take $\mu_0$ to be the counting measure on $\Qp\slash\Zp$, which is the unique Haar measure that gives $\Zp$ unit measure. %
Use the canonical inclusion map that takes $\Zp$ into $\Qp$ to define the dual pairing $\langle\cdot, \cdot\rangle$.  Namely, \[\langle\cdot, \cdot\rangle \colon \Zp \times (\Qp\slash\Zp) \to \mathds S^1\quad \text{by}\quad \langle x, \Gdmap{y}\rangle = \chi(xy).\] For any $(\Gdmap{x}, w, y)$ in $(\Qp\slash\Zp)\times \Zp \times \Zp$, the product $wy$ is in $\Zp$, and so \begin{align*}\chi((x+w)y)= \chi(xy)\chi(w y) = \chi(xy).\end{align*} The definition of the pairing is, therefore, independent of the choice of representative.

Define the Fourier transform $\F_0$ that takes $L^2(\Zp)$ to $L^2(\Qp\slash\Zp)$ and the inverse Fourier transform $\F^{-1}_0$ that takes $L^2(\Qp\slash\Zp)$ to $L^2(\Zp)$ to be given for any $f$ in $L^2(\Zp)$ and any $\tilde{f}$ in $L^2(\Qp\slash\Zp)$ by \begin{align*}(\F_0f)(y) = \int_{\Zp}\langle -x,\y\rangle f(x)\,{\rm d}\mu(x)
\quad \text{and} \quad (\F_0^{-1}\tilde{f})(x) = \int_{\Qp\slash\Zp}\langle x,\y\rangle \tilde{f}(\y)\,{\rm d}\mu_0(\y).\end{align*}

Take $\beta$ to be the quantity \begin{equation}\label{beta}\beta = \frac{p^{b+1}-1}{p^b(p-1)}\end{equation} and $\M_{b,0}$ to be the multiplication operator that acts on any compactly supported function $f$ with domain $\Qp\slash\Zp$ by \[(\M_{b,0}f)(\y) = (|\y|^b - \beta^{-1})f(\y).\]  The Vladimirov-Kochubei operator is the self-adjoint closure of the operator that is given for any function $f$ in $SB(\Zp)$ by \[(\Delta_{b,0}f)(x) = (\F^{-1}\M_{b,0}\F f)(x).\]  The diffusion equation in $\Zp$ is the equation given by \eqref{Intro:HeatEquation}, but with the new operator $\Delta_{b,0}$ replacing $\Delta_b$.  The Fourier transform of the fundamental solution $\rho$ to the diffusion equation on $\Zp$ is $\phi$, where \begin{equation}\label{Limiting:ZP:CharFun}\phi(t,\y) = \e^{-D\left(|\y|^b - \beta^{-1}\right)t}.\end{equation} The ordered collection of functions $(\rho(t,\cdot))_{t>0}$ forms a convolution semigroup of probability density functions on $\Zp$ that gives rise to a probability measure $\Pb$ on $D([0,\infty)\colon \Zp)$ and the triple $(D([0,\infty)\colon \Zp), \Pb, Y)$ is a Brownian motion in $\Zp$.


\section{The discrete settings}\label{Sec:TheDiscreteSettings}



\subsection{The discrete state spaces}

Take $\Gm$ to be the discrete group $\Zp\slash p^m\Zp$ and $\Gmdual$ its Pontryagin dual, the discrete group $p^{-m}\Zp\slash\Zp$.  It will be useful to write the elements of these groups in a way that both clearly indicates their membership and also their relationship to the elements of $\mathds Q_p$.  For this reason, define the projections $\Gmap{\cdot}_m$ and $\Gdmap{\cdot}_m$ for any $x$ in $\mathds Z_p$ and for any $y$ in $p^{-m}\Zp$ respectively by \[\x_m = x + p^m\Zp \quad \text{and} \quad \y_m = y + \Zp,\] so that \[\Gmap{\cdot}_m \colon \Zp \to \Gm = \Zp\slash p^m\Zp \quad \text{and} \quad \Gdmap{\cdot}_m\colon p^{-m}\Zp \to \Gmdual = p^{-m}\Zp\slash\Zp.\] The $p$-adic absolute value induces a norm on both $\Gm$ and on $\Gmdual$, respectively, by \[\begin{cases}\big|\x_m\big|_m = |x| & \text{if } \x_m \ne \Gmap{0}_m\\\big|\Gmap{0}_m\big|_m = 0 & \end{cases} \quad \text{and} \quad \begin{cases}\big|\y_m\big|_m = |y| & \text{if } \y_m \ne \Gdmap{0}_m\\\big|\Gdmap{0}_m\big|_m = 0. & \end{cases}\] Henceforth, suppress the $m$ in the notation whenever it is unambiguous to do so.

For any $\ell$ in $\{1, \dots, m\}$, take $\Bm(\ell)$ and $\Sm(\ell)$, respectively, to be the ball and circle of radius $p^{\ell-m}$ in $\Gm$ with center $p^m\Zp$, and take \[\Bm(0) = \Sm(0) = p^m\Zp.\] Take $\mu_m$ to be the Haar measure on $\Gm$ that is normalized to give $\Gm$ unit measure, so that for any $k$ in $\{1, 2, 3, \dots, m\}$, the volume of balls and circles are given by\begin{align}\label{DiscreteSettings:Eq:BallA}\begin{cases}\Vol(\Bm(k)) = p^{k-m} & \text{}\\ \Vol(\Bm(0)) = p^{-m} &  \text{} \end{cases} \quad \text{and} \quad \begin{cases}\Vol(\Sm(k)) = \big(1 - p^{-1}\big)p^{k-m} &  \text{}\\ \Vol(\Sm(0)) = p^{-m}. & \text{}\end{cases}\end{align}

For any $\ell$ in $\{1, \dots, m\}$, take $\Bmdual(\ell)$ and $\Smdual(\ell)$, respectively, to be the ball and circle in $\Gmdual$ of radius $p^\ell$ with center $\Zp$, and \[\Bmdual(0) = \Smdual(0) = \Zp.\] Take $\tilde{\mu}_m$ to be the Haar measure on $\Gmdual$ that is normalized to give $\Zp$ unit measure.  This measure is also the counting measure, so for any $k$ in $\{1, 2, 3, \dots, m\}$, %
\begin{align}\label{DiscreteSettings:Eq:BallB}
\begin{cases}\Vol(\Bmdual(k)) = p^{k} & \text{}\\ \Vol(\Bmdual(0)) = 1 &  \text{} \end{cases} \quad \text{and} \quad \begin{cases}\Vol(\Smdual(k)) = \big(1 - p^{-1}\big)p^{k} &  \text{}\\ \Vol(\Smdual(0)) =1. & \text{}%
\end{cases}
\end{align}

The character $\chi$ on $\mathds Q_p$ induces a dual pairing $\langle \cdot, \cdot\rangle$  on $\Gm\times\Gmdual$ that is given for any $(\x, \y)$ in $\Gm\times\Gmdual$ by \[\langle \x, \y\rangle = \chi(xy).\]  As before, the pairing makes use of a specific choice of representative for the group elements, but since $\chi$ is additive and rank $0$, the pairing is independent of the choice of representative.  The Fourier transform $\F_m$ from $L^2(\Gm)$ to $L^2(\Gmdual)$ and its inverse $\F_m^{-1}$ act, respectively, on any functions $f$ and $g$ by \[(\Fm f)(\y) = \int_{\Gm} \chi(xy)f(\x)\,{\rm d}\mu_m(\x) \quad \text{and} \quad (\Fm^{-1}g)(\x) = \int_{\Gmdual} \chi(-xy)g(\Gdmap{x})\,{\rm d}\tilde{\mu}_m(\y).\]  Both $\Fm$ and $\Fm^{-1}$ are unitary with respect to the $L^2$-norm.

For any set $A$ that is contained in a set $B$, denote by $\mathds 1_A$ the indicator function on $A$, the function that is given by \[\mathds 1(x) = \begin{cases}1 & \text{if } x\in A\\0 & \text{if }x \in B\setminus A.\end{cases}\]  Lemma~\ref{Construction:Eq:indicatori} and Lemma~\ref{Construction:Eq:Dualindicatori} are useful for calculating integrals that involve the dual pairing.  Their proofs follow from the proofs of the corresponding statements in the $\mathds Q_p$ setting.

\begin{lemma}\label{Construction:Eq:indicatori}
For any $i$ in $\{1, \dots, m\}$ and $\y$ in $\Gmdual$,
\begin{align*}
\int_{\Bm(i)}\chi(xy)\,{\rm d}\mu_m(\x) &= \Vol(\Bm(i))\cdot\mathds{1}_{\Bmdual(m - i)}(\y) \quad \text{and} \quad \int_{\Bm(0)}\chi(xy)\,{\rm d}\mu_m(\x) = p^{-m}.
\end{align*}
\end{lemma}

\begin{lemma}\label{Construction:Eq:Dualindicatori}
For any $i$ in $\{1, \dots, m\}$ and $\x$ in $\Gm$, 
\begin{align*}
\int_{\Bmdual(i)}\chi(xy)\,{\rm d}\tilde{\mu}_m(\y) = \Vol(\Bmdual(i))\cdot\mathds{1}_{\Bm(m - i)}(\x) \quad \text{and} \quad \int_{\Bmdual(0)}\chi(xy)\,{\rm d}\tilde{\mu}_m(\y) = p^{-m}.
\end{align*}
\end{lemma}



\subsection{Construction of the $\Gm$-valued processes}

Take $\X^{(m)}$ to be the $\Gm$-valued abstract random variable with probability mass function $\rho_{\X^{(m)}}$ that is constant on each circle in $\Gm$ and that satisfies the equalities
\[ 
\begin{cases} 
      \text{Prob}(\X^{(m)} \in \Sm(\ell) ) = \frac{C_m}{p^{\ell b}} & \text{if } \ell \in \{1,2,\ldots,m\}\\
      \text{Prob}(\X^{(m)} = \Gmap{0}) = 0. &
\end{cases}
\]
To simplify notation, henceforth write $X$ rather than $\X^{(m)}$ unless it is necessary to specify the natural number $m$.   The equality 
\begin{align*}
1 &= \Bigg\{\frac{1}{p} + \frac{1}{p^{2b}} + \cdots + \frac{1}{p^{mb}} \Bigg\} \cdot c_m = \frac{1}{p^{mb}} \cdot \frac{p^{mb}-1}{p^b-1} \cdot c_m
\end{align*}
implies that
\[c_m = \frac{p^{mb}}{p^{mb}-1}\big(p^b-1\big).\]  %
The constancy of $\rho_{X}$ on each circle implies that 
\[ 
\begin{cases} 
      \rho_X(\x) = \frac{1}{\Vol(\Sm(k))}\cdot \frac{c_m}{p^{kb}} & \text{if } \left|\x\right| = p^{-m+k}\\
       \rho_X(\Gmap{0}) = 0.&
\end{cases}
\]

To compress notation, henceforth write ${\rm d}\mu_m(\x)$ and ${\rm d}\tilde{\mu}_m(\y)$ as $\d\x$ and $\d\y$, respectively.  Denote by $\phi_X$ the characteristic function of $X$ and use the decomposition of $\Zp$ into a countable union of circles to obtain the equalities%
\begin{align*}
\phi_X(\y) & = (\Fm\rho_X)(\y)\\
&= \int_{\Gm} \chi(xy)\rho_X(\x)\,\d\x\\
&= \int_{\Gm} \chi(xy)\rho_X(\x)\,\d\x + \sum_{\ell = 1}^m\int_{\Sm(-m+\ell)} \chi(xy)\rho_X(\x)\,\d\x.
\end{align*}
For each $i$ in $\{0, 1, \dots, m\}$, simplify notation by writing \[\rho(i) = \begin{cases}0 &\text{if } i = 0\\\frac{1}{\text{Vol}(\Sm(i))} \cdot \frac{c_m}{p^{i b}} & \text{if }i \in \{1, \dots, m\}\end{cases}\] to obtain for any $\y$ that is not equal to $\Gdmap{0}$ the equalities
\begin{align}\label{Construction:Eq:PhiX}
\phi_X(\y) & = \sum_{i=1}^m\rho(i)\int_{\Sm(i)}\chi(xy)\,\d\x%
\notag\\& = \sum_{i=1}^m \rho(i)\Bigg(\int_{\Bm(i)} \chi(xy)\,\d\x - \int_{\Bm(i)}\chi(xy)\,\d\x\Bigg)\notag\\
 &= -\rho(1) \int_{\Bm(-m)}\chi(xy)\,\d\x + \sum_{i=1}^{m-1}(\rho(i)-\rho(i+1)) \int_{\Bm(i)}\chi(xy)\,\d\x\notag\\& 
\qquad\qquad + \rho(m)\int_{\Bm(m)}\chi(xy)\,\d\x.
\end{align}

\begin{proposition}\label{Construction:Prop:PhiFormula}
The characteristic function for $X$, $\phi_X$, is given for any $\y$ not equal to $\Gdmap{0}$ by 
\begin{align}
\phi_X(\y) = 1 - \dfrac{\beta\Big(|\y|^b\big(1+ \frac{1}{p^{mb}}\big) - \beta^{-1}\Big)}{p^{mb}}+\frac{1}{p^{mb}(p^{mb}-1)}\left(1 - \dfrac{\beta|\y|^b}{p^{mb}}\right),
\end{align}
and \[\phi_X(\Gdmap{0}) = 1.\]
\end{proposition}

\begin{proof}
For any $\y$ in $\Gmdual$, either $|\y|$ is equal to $\Gdmap{0}$, or there is a $k$ in $\{1, 2, \dots, m\}$ so that $|\y|$ is equal to $p^{k}$. Since $\phi_X$ is the Fourier transform of a probability density function, it takes on the value $1$ at $\Gdmap{0}$.  Use \eqref{Construction:Eq:PhiX} to obtain the equalities
\begin{align*}
\phi_X(\y) & = -\rho(1)\cdot p^{-m} + \sum_{i=1}^{m-1}\Big\{(\rho(i) - \rho(i+1))\Vol(\Bm(i))\mathds{1}_{\Bmdual(m-i)}(\y)\Big\}\\
 &\qquad + \rho(m) \cdot \Vol(\Bm(m)) \cdot \mathds{1}_{\Gdmap{0}}(\y).
\end{align*} %
For each $i$ in $\{1, 2, \dots, m\}$, %
\begin{align*}
\big(\rho(i) - \rho(i +1)\big) \cdot \Vol\big(\Bm(i)\big)
= \frac{c_m}{p^{ib}}\cdot\frac{p}{p-1}\cdot\frac{p^{b+1} -1}{p^{b+1}}.
\end{align*}
Furthermore,
\begin{align*}
\rho(1) \cdot p^{-m}= \Prob\big(X\in\Sm(1)\big)\cdot\frac{1}{\Vol\big(\Sm(1)\big)}\cdot p^{-m}
= \frac{c_m}{p^b} \cdot \frac{1}{p-1},
\end{align*}
and so
\begin{align*}
\phi_X(\y)&= - \frac{c_m}{p^b} \cdot \frac{1}{p-1} + c_mC(p,b)\Bigg\{\frac{1}{p^b}\cdot \mathds{1}_{\Bmdual(m-1)}(\y) + \frac{1}{p^{2b}} \cdot \mathds{1}_{\Bmdual(m-2)}(\y) \\&\qquad\qquad + \cdots + \frac{1}{p^{i b}} \cdot \mathds{1}_{\Bmdual(m-i)}(y) + \cdots + \frac{1}{p^{(m-1)b}} \cdot \mathds{1}_{\Bmdual(1)}(\y) \Bigg\} \\&\qquad\qquad + \frac{1}{1-\frac{1}{p}} \cdot \frac{c_m}{p^{mb}} \cdot \mathds{1}_{\Bmdual(0)}([y]),
\end{align*}
where $C(p,b)$ denotes the quantity \[C(p,b) = \frac{p}{p-1}\cdot\frac{p^{b+1} -1}{p^{b+1}}.\]

Take $|\y|$ to be equal to $p^k$, so that 
\begin{align*}
\phi_X(\y)&= - \frac{c_m}{p^b} \cdot \frac{1}{p-1} + c_mC(p,b)\cdot \Bigg \{\frac{1}{p^b} + \frac{1}{p^{2b}} + \cdots + \frac{1}{p^{(m-k)b}}\Bigg\}\\
&=\frac{p^{mb}}{p^{mb}-1} \Bigg \{ 1 - \frac{p^{b+1}-1}{p^b(p-1)} \cdot \frac{p^{kb}}{p^{mb}} \Bigg \}=\frac{p^{mb}}{p^{mb}-1} \Bigg \{ 1 - \frac{p^{b+1}-1}{p^b(p-1)} \cdot \frac{|\y|^b}{p^{mb}} \Bigg\}.
\end{align*}
Use the series expansion%
\begin{align*}
\frac{p^{mb}}{p^{mb} - 1} = 1 + \frac{1}{p^{mb}} + \frac{1}{p^{2mb}} + \frac{1}{p^{3mb}} + \cdots
\end{align*} %
to rewrite $\phi_X(\y)$ for any $\y$ not equal to $\Gdmap{0}$ as
\begin{align*}
\phi_X(\y) &= \left(1 + \frac{1}{p^{mb}} + \frac{1}{p^{2mb}} + \frac{1}{p^{3mb}} + \cdots\right)\left(1 - \dfrac{\beta|\y|^b}{p^{mb}}\right)\\
&= 1 - \dfrac{\beta\Big(|\y|^b\big(1+ \frac{1}{p^{mb}}\big) - \beta^{-1}\Big)}{p^{mb}}+\frac{1}{p^{mb}(p^{mb}-1)}\left(1 - \dfrac{\beta|\y|^b}{p^{mb}}\right)
\end{align*}
where $\beta$ is given by \eqref{beta}.
\end{proof}

The Kolmogorov Extension theorem guarantees that there is a measure $\Pb_m^\ast$ on $F(\mathds N_0\colon \Gm)$ and a stochastic process $(F(\mathds N_0\colon \Gm), \Pb_m^\ast, S)$ so that the increments of $S$ are independent and have the same law as $X$.  For any natural number $n$, take $(X_i)$ to be the sequence of increments of this process, so that for each $n$ in $\mathds N_0$, \[S_n = S_0 + X_1 + \cdots + X_n\] and $S_0$ is almost surely equal to $\Gmap{0}$.  Take $\rho(n,\cdot)$ to be the law for $S_n$.  To compress notation, for any $i$ in $\{1, \dots, m\}$, write \[\phi(i) = \phi_X(\y),\quad \text{where}\quad \y = p^i,\] and write $\phi(0)$ rather than $\phi_X(\Gdmap{0})$.

\begin{proposition}\label{Construction:Prop:PDFcalculation}
For any natural number $n$, \begin{equation}\label{Construction:PropEq:PDFcalculation}\rho(n,\x) = \sum_{i=0}^{m-1}\Big(\phi(i)^n - \phi(i+1)^n\Big)p^i\mathds{1}_{\Bm(m-i)}(\x)+ \phi(m)^n\mathds{1}_{\Bm(0)}(\x).\end{equation}
\end{proposition}

\begin{proof}
Take the inverse Fourier transform of the $n$-fold product of the characteristic function of $X$ to obtain the equalities 
\begin{align*}
\rho(n,\x)&= \big (\F^{-1}(\phi_X(\cdot)^n\big)(\x)\\
&= \int_{p^{-m}\Zp\slash\Zp} \chi(xy) \phi_X(\y)^n\,\d\y
= \sum_{i=0}^m \phi(i)^n\int_{\Smdual(i)} \chi(xy)\,\d\y.
\end{align*}

Rearrange terms to obtain the equalities
\begin{align*}
\rho(n,\x) &= \int_{\Bmdual(0)}\chi(xy)\,\d\y + \sum_{i=1}^m\phi(i)^n\Bigg(\int_{\Bmdual(i)} \chi(xy)\,\d\y - \int_{\Bmdual(i-1)} \chi(xy)\,\d\y\Bigg)\\
& = (1 - \phi(1)^n)\int_{\Bmdual(0)} \chi(xy)\,\d\y \\&\qquad\qquad + \sum_{i=1}^{m-1}(\phi(i)^n - \phi(i+1)^n)\int_{\Bmdual(i)}\chi(xy)\,\d\y + \phi(m)^n\int_{\Bmdual(m)}\chi(xy)\,\d\y,
\end{align*}
which together with Lemma~\ref{Construction:Eq:Dualindicatori} implies \eqref{Construction:PropEq:PDFcalculation}.
\end{proof}


\subsection{Isometric embeddings of the discrete state spaces}

Take $\Gamma_m$ to be the injection from $\Gm$ to $\Zp$ that is given for any $x$ in $\Zp$ by %
\begin{equation}\label{Gamma-m}
\Gamma_m(\x) = \sum_{i\in\{0, \dots, m\}} a_x(i)p^{i}.
\end{equation} %
For any $g$ in $\Gm$, there is an $x$ in $\Zp$ so that $g$ is equal to $\x$.  For any other such element $x^\prime$ of $\Zp$, $\Gamma_m(\x)$ is equal to $\Gamma_m(\Gmap{x^\prime})$, and so this injection is independent of choice of representative.  The function $\Gamma_m$ is an isometry because the metric on $\Zp$ is an ultrametric, however, $\Gamma_m$ is not a group homomorphism.  Because of this, it is helpful to first construct random walks in the $\Gm$, and then use $\Gamma_m$ together with a time scaling to transform the $\Gm$-valued random walks into random walks on a discrete subset of $\Zp$.


\section{Moments of the processes on the $\Gm$}\label{Sec:Moments}


\subsection{Preliminary estimates}

To work with the formula for $\phi_X$ that Proposition~\ref{Construction:Prop:PhiFormula} provides requires some determination of the sign of various associated quantities.  The first required estimate, Proposition~\ref{Prim:Prop:Boundonalphaoverpb}, is already given in earlier works \cite{WJPA}, although more is now needed.

\begin{proposition}\label{Prim:Prop:Boundonalphaoverpb} 
For any positive real number $b$, both $\beta - 1$ and $\frac{\beta}{p^b}$ are in $(0,1)$.
\end{proposition}

It is helpful to list some formulas that the proof of Theorem~\ref{Moment:Thm:Main} requires and analyze the positivity of the various terms.  For any $i$ in $\{1, \dots, m\}$,
\begin{align}\label{Momement:Eq:CharacterDiff}
    \phi(i) - \phi(i + 1) 
&= \beta(p^b - 1)\frac{p^{ib}}{p^{mb}} + \beta(p^b - 1)\frac{p^{ib}}{p^{mb}}\frac{1}{p^{mb}} + \beta(p^b - 1)\frac{p^{ib}}{p^{mb}}\frac{1}{p^{mb}(p^{mb}-1)},
\end{align}
which shows that $(\phi(i))$ is increasing in $i$.  The equality
\[\phi(m) = (1-\beta)\left(1 + \frac{1}{p^{mb}} + \frac{1}{p^{mb}(p^{mb}-1)}\right)\] implies that $\phi(m)$ is negative, however a similar calculation shows that $\phi(m-1)$ is positive.  Notice that the formula for $\phi(i)$ only agrees with $\phi(0)$ in limit, as the latter is equal to $1$.  Write $\psi(0)$ to extend the formula for $\phi(i)$, where $i$ is non-zero, to the case when $i$ is $0$ and obtain the equality 
\[\psi(0) = 1 - p^{-mb}\left(\beta- 1 + \beta p^{-mb}\right) + p^{-mb}(p^{mb}-1)^{-1}(1-\beta p^{-mb}).\] A simple calculation shows that the quantity $\psi(0)$ is bounded above by $1$ if and only if $\beta$ is bounded below by $1$, which it is.  Since $\phi(i)$ is decreasing in $i$, $\phi(i)$ is in $[0,1]$ for every $i$ in $\{1, \dots, m-1\}$.

For any $m$ and every $i$ in $\{1, \dots, m-1\}$, %
\begin{align}\label{Moments:Eq:CharFunEqpowN}
    \phi(i)^n&= \left(1 - \dfrac{\beta\Big(p^{ib}\big(1+ \frac{1}{p^{mb}}\big) - \beta^{-1}\Big)}{p^{mb}}+\frac{1}{p^{mb}(p^{mb}-1)}\left(1 - \dfrac{\beta p^{ib}}{p^{mb}}\right)\right)^n,
\end{align} %
and so
\begin{align}\label{Moment:Char:Bound}
    \phi(i)^n&\leq \left(1 - \dfrac{\beta\Big(p^{ib} - \beta^{-1}\Big)}{p^{mb}}+\frac{1}{p^{mb}(p^{mb}-1)}\right)^n= \left(1 - \dfrac{\beta p^{ib} - \frac{p^{mb}}{p^{mb}-1}}{p^{mb}}\right)^n.
\end{align}
Finally, since $\phi(0)$ is not given by the same formula as the other $\phi(i)$, it will be important to separately determine the quantity $\phi(0)^n-\phi(1)^n$, which is given by
\begin{align}\label{Moments:Eq:Exceptional}
\phi(0) - \phi(1)^n & = 1 - \phi(1)^n\\
& \leq 1 - \left(1 - \dfrac{\beta p^{b}\big(1+ \frac{1}{p^{mb}}\big)}{p^{mb}}\right)^n.
\end{align}


\subsection{Moment estimates}

Denote by ${\mathds E}^\ast_m$ the expected value with respect to the measure ${\rm P}^\ast_m$.  Henceforth, take $M(p,b)$ to be a natural number with the property that if $m$ is greater than $M(p,b)$, then \[\beta \ge \frac{p^{mb}}{p^{mb}-1}.\]  Some basic manipulations reveal that the given inequality is valid as long as \[m \ge \frac{1}{b} \log_p\left(\frac{p^{b+1}-1}{p^b-1}\right),\] so if $b$ is equal to $1$, and $p$ is greater than $2$, then the condition is non-restrictive.  If $p$ is 2, then $m$ must be at least 2.  Since the arguments below are simplified if $m$ is at least 2, for sake of convenience take $M(p,b)$ to be given, henceforth, by \[M(p,b) = 1+ \left\lceil\frac{1}{b} \log_p\left(\frac{p^{b+1}-1}{p^b-1}\right)\right\rceil,\] where $\ceil{\cdot}$ represents the \emph{ceiling function}.  Note that small values of $b$ require $m$ to be larger.  The intuition should be that paths more rapidly exit balls of a fixed radius for smaller values of $b$, as earlier work on first exit times demonstrates both in the single and higher dimensional settings \cite{Weisbart:2021, RW:JFAA:2023}.

\begin{theorem}\label{Moment:Thm:Main}
For any natural number $n$ and any positive real number $r$ in $(0,b)$, there are constants $K$ and $C$ that are independent of $m$ so that for any $m$ that is greater than $M(p,b)$, %
\[
{\mathds E}_m^\ast[|S_n|^r] \leq Kn^{\frac{r}{b}}p^{-mr} +  \left(1 - \left(1 - \dfrac{\beta p^{b}\big(1+ \frac{1}{p^{mb}}\big)}{p^{mb}}\right)^n\right)\frac{p^r(p-1)}{p^{r+1}-1}.
\]
\end{theorem}

\begin{proof}
Use the expression for $\rho(n,[x])$ given by Proposition~\ref{Construction:Prop:PDFcalculation} to obtain the equalities %
\begin{align}\label{Moment:FullFirstEquality}
{\mathds E}^\ast_m[|S_n|^r] &=\int_{\Gm} |[x]|^r \rho(n,[x])\,\d\x\notag\\
&= \sum_{i=0}^{m-1}\Big(\phi(i)^n - \phi(i+1)^n\Big)p^i\int_{\Gm}|[x]|^r\mathds{1}_{\Bm(m-i)}([x])\,\d\x. 
\end{align}
For each $i$ in $\{0,1, \dots, m\}$, use the decomposition of an integral over $\Gm$ into an integral over circles to obtain the equalities
\begin{align}\label{Moment:Eq:First}
\int_{\Gm}|[x]|^r\mathds{1}_{\Bm(m-i)}([x])\,\d\x & = \sum_{\ell = i}^{m}\int_{\Gm}|[x]|^r\mathds{1}_{\Sm(m-\ell)}([x])\,\d\x\notag\\
&= \left(1-\frac{1}{p}\right)p^{-r(m-1)}p^{-(m-1)}\left\{\frac{p^{r(m-i)}p^{(m-i)}-1}{p^{r}p^{1}-1}\right\}.
\end{align}
Multiply by $p^i$ to see that 
\begin{align}\label{Moment:Eq:Second}
p^i\int_{\Gm}|[x]|^r\mathds{1}_{\Bm(-i)}([x])\,\d\x = \left(p-1\right)p^rp^i\left\{\frac{p^{-ir}p^{-i}-p^{-rm}p^{-m}}{p^{r+1}-1}\right\}.
\end{align}

Equations~\eqref{Moment:Eq:First} and \eqref{Moment:Eq:Second} together imply that
\begin{align}\label{Moment:FullFirstEquality}
{\mathds E}[|S_n|^r] %
 & = \sum_{i=1}^{m-1}\Big(\phi(i)^n - \phi(i+1)^n\Big)\left(p-1\right)p^rp^i\left\{\frac{p^{-i(r+1)}-p^{-m(r+1)}}{p^{r+1}-1}\right\}\notag\\%
 &\qquad\qquad + \Big(1 - \phi(1)^n\Big)\left(p-1\right)p^r\left\{\frac{1-p^{-m(r+1)}}{p^{r+1}-1}\right\}.
\end{align}
Denote by $I(m,n)$ the quantity
\begin{align*}
I(m,n) &= \sum_{i=1}^{m-1}\Big(\big(\phi(i)\big)^n - \big(\phi(i+1)\big)^n\Big)\left(p-1\right)p^rp^i\left\{\frac{p^{-i(r+1)}-p^{-r(m+1)}}{p^{r+1}-1}\right\}\\
&< \sum_{i=1}^{m-1}\Big(\big(\phi(i)\big)^n - \big(\phi(i+1)\big)^n\Big)\left(p-1\right)p^rp^i\left\{\frac{p^{-i(r+1)}}{p^{r+1}-1}\right\}\\
&= \frac{p^r(p-1)}{p^{r+1}-1}\sum_{i=1}^{m-1}\Big(\phi(i)^n - \phi(i+1)^n\Big)p^{-ir}.
\end{align*}
Use the fact that $(\phi(i)^n)$ is a decreasing finite sequence in $[0,1]$ in the index $i$ to write the difference of successive terms as%
\begin{align*}
\phi(i)^n - \phi(i+1)^n & = n\int_{\phi(i+1)}^{\phi(i)} s^{n-1}\,{\rm d}s.
\end{align*}
Since the integrand is an increasing function on $[0,1]$, 
\begin{align}\label{Moment:Eq:DifferenceEstim}
\phi(i)^n - \phi(i+1)^n & < n\phi(i)^{n-1}\big(\phi(i) - \phi(i+1)\big)\notag\\%
& = c(m)n\phi(i)^{n-1}\frac{p^{ib}}{p^{mb}},
\end{align}
where \[c(m) = \beta(p^b-1)\left(1 + \frac{1}{p^{mb}} + \frac{1}{p^{mb}(p^{mb}-1)}\right).\] 
Equations \eqref{Moment:Char:Bound} and \eqref{Moment:Eq:DifferenceEstim} together imply that 
\begin{align}\label{Moment:Imn:ExpressAsxi}
I(m,n) &\leq \frac{c(m)p^r(p-1)}{p^{r+1}-1}\sum_{i=1}^{m-1}np^{-ir}\phi(i)^{n-1}\frac{p^{ib}}{p^{mb}}\notag\\
& = \frac{c(m)p^r(p-1)}{p^{r+1}-1}\sum_{i=1}^{m-1}np^{-ir}\left(1 - \dfrac{\beta p^{ib} - \frac{p^{mb}}{p^{mb}-1}}{p^{mb}}\right)^{n-1}\frac{p^{ib}}{p^{mb}}
\end{align}

For each $i$ in $\{0, 1, \dots, m-1\}$ write \[x_i = \dfrac{\beta p^{ib} - \frac{p^{mb}}{p^{mb}-1}}{p^{mb}}.\]  As long as $m$ is taken to be large enough so that \[\beta \ge \frac{p^{mb}}{p^{mb}-1},\] $x_0$ is non-negative, and so this is valid under the hypotheses of the theorem.  For each interval $I(i)$ with endpoints $x_{i-1}$ and $x_{i}$, the length of $I(i)$, denoted $\Delta(i)$, is given by \[\Delta(i) = x_{i} - x_{i-1} = \frac{1}{p^b}\beta(p^b-1)\frac{p^{ib}}{p^{mb}}.\]  Use the expressions for $x_i$ and $\Delta(i)$ to rewrite \eqref{Moment:Imn:ExpressAsxi} and obtain the inequality%
\begin{align}\label{Moment:Imn:withr}
I(m,n) &\leq \frac{c(m)p^r(p-1)}{\beta p^b(p^b-1)(p^{r+1}-1)}\sum_{i=1}^{m-1}np^{-ir}\left(1 - x_i\right)^{n-1}\Delta(i).
\end{align}
Rewrite $p^{-ir}$ in terms of $x_i$ to obtain the equality \[p^{-ir} = \left(x_i + \frac{1}{p^{mb}-1}\right)^{-\frac{r}{b}}p^{-rm}\beta^{\frac{r}{b}} < x_i^{-\frac{r}{b}}p^{-rm}\beta^{\frac{r}{b}},\] which implies together with \eqref{Moment:Imn:withr} that
\begin{align}\label{Moment:Imn:finalImn}
I(m,n) &\leq Kp^{-rm}\sum_{i=1}^{m-1}nx_i^{-\frac{r}{b}}\left(1 - x_i\right)^{n-1}\Delta(i),
\end{align}
where \[K = \frac{c(m)p^r(p-1)}{p^b(p^b-1)(p^{r+1}-1)}\beta^{\frac{r-b}{b}}.\]

The sum in \eqref{Moment:Imn:finalImn} is a lower Riemann sum approximation of the integral 
\begin{align*}
\int_{x_1}^{x_{m-1}} nx^{-\frac{r}{b}}(1-x)^{n-1}\,{\rm d}x &<  \int_0^{1} nx^{-\frac{r}{b}}(1-x)^{n-1}\,{\rm d}x.
\end{align*}
Follow the earlier arguments in the $\mathds Q_p$ and local field settings \cite{WJPA, W:Expo:24} to obtain the inequality
\begin{equation}
\sum_{i=1}^{m-1}nx_i^{-\frac{r}{b}}\left(1 - x_i\right)^{n-1}\Delta(i-1) \leq n^{-\frac{b-r}{b}}\Big(\frac{n+\frac{b-r}{b}}{n}\Big)^{\frac{r}{b}},
\end{equation}
which together with \eqref{Moment:Imn:finalImn} implies that 
\begin{equation}
I(m,n) \leq Kp^{-rm}n^{\frac{r}{b}}\Big(\frac{n+\frac{b-r}{b}}{n}\Big)^{\frac{r}{b}}.
\end{equation}
Take $C$ to be equal to $K\Big(\frac{n+\frac{b-r}{b}}{n}\Big)^{\frac{r}{b}}$ to obtain the inequality 
\begin{equation}
I(m,n) \leq Cp^{-rm}n^{\frac{r}{b}}.
\end{equation}

The inequality \eqref{Moments:Eq:Exceptional} gives an upper bound for the right hand term in \eqref{Moment:FullFirstEquality}, namely%
\begin{align*}
\Big(1 - \phi(1)^n\Big)\left(p-1\right)p^r\left\{\frac{1-p^{-m(r+1)}}{p^{r+1}-1}\right\}\leq \left(1 - \left(1 - \dfrac{\beta p^{b}\big(1+ \frac{1}{p^{mb}}\big)}{p^{mb}}\right)^n\right)\frac{p^r(p-1)}{p^{r+1}-1}.
\end{align*}

\end{proof}


\section{Convergence of the embedded processes}\label{Sec:Convergence}



\subsection{The embedded processes}

Given any null sequence $(\tau_m)$, a sequence of \emph{time scales}, take $(\iota_m)$ to be a sequence of functions so that for each $m$, \[\iota_m \colon \mathds N_0\times \Gm \to [0,\infty)\times \Zp\quad \text{by} \quad \iota(n,\x) = (\tau_mn, \Gamma_m(\x)).\]  For typographical efficiency, take $\lambda_m$ to be the reciprocal of $\tau_m$.  The function $\iota_m$ acts on the stochastic process $\big(F(\mathds N_0\colon \Gm), {\rm P}^\ast_m, S\big)$ to produce a continuous time process in the following way.  For any real number $t$, define the abstract random variable $\Y_t$ by \[\Y_t = \Gamma_m\Big(S_0 + X_1 + \dots + X_{\floor{t\lambda_m}}\Big).\]  

The Kolmogorov Extension theorem guarantees the existence of a stochastic process $(F([0,\infty)\colon \Zp), {\rm P}_m, Y)$ that has the same finite dimensional distributions as $\Y$.  Equivalently, directly define ${\rm P}_m$ on any simple cylinder set $C(h)$ by \begin{equation}\label{Convergence:Eq:neMeas}{\rm P}_m(C(h)) = {\rm P}^\ast_m\left(\bigcap_{i\in\{0, 1, \dots, \ell(h)\}}\left(S_{\floor{e_h(i)\lambda}}\right)^{-1}(\Gamma^{-1}_m(U(i))\right).\end{equation} 

Take $\sigma$ to be any positive real number, and specialize $(\tau_m)$ so that \[\lambda_m = \sigma p^{mb}.\] Take ${\mathds E}_m$ to be the expected value with respect to the measure $\Pb_m$.  Denote by $M^\prime(p,b)$ the quantity \[M^\prime(p,b) = \frac{1}{b}\log_p(2\sqrt{2}p^b).\]  A straightforward calculation shows that if $m$ is at least $M^\prime(p,b)$, then \[\beta p^b\Big(1+\frac{1}{p^{mb}}\Big) < p^{mb}.\]

\begin{lemma}\label{Convergence:lemma:ExtraTerm}
For any natural number $m$ that is greater than $M^\prime(p,b)$ and any real number $s$ in $(0,1)$, there is a constant $K$ so that for any $t$ in $[0, \infty)$,
\begin{align*}
\left(1 - \left(1 - \dfrac{\beta p^{b}\big(1+ \frac{1}{p^{mb}}\big)}{p^{mb}}\right)^{t\sigma p^{mb}}\right) \leq Kt^s.
\end{align*}
\end{lemma}

\begin{proof}
For any natural number $n$ and any $x$ in $[0,n]$, %
\begin{equation}\label{Convergence:Eq:BinIneqA}
1-x \leq \Big(1-\frac{x}{n}\Big)^n\quad\text{and so}\quad x \geq 1-\Big(1-\frac{x}{n}\Big)^n.
\end{equation}  %
Change variables and take $m$ to be large enough so that \[\beta p^{b}\Big(1+ \frac{1}{p^{mb}}\Big) \leq p^{mb}\] to see that \eqref{Convergence:Eq:BinIneqA} implies that 
\begin{equation}\label{Convergence:Eq:BinIneqB}
2p^b\beta\sigma t \geq \left(1 - \left(1 - \dfrac{\beta p^{b}\big(1+ \frac{1}{p^{mb}}\big)}{p^{mb}}\right)^{t\sigma p^{mb}}\right).
\end{equation}  %
The quantity on the right hand side of \eqref{Convergence:Eq:BinIneqB} is bounded above by $1$, and for any $t$ in $[0,1]$, $t^s$ is greater than or equal to $t$, so
\begin{equation}\label{Convergence:Eq:BinIneqC}
\max(2p^b\beta\sigma, 1)t^s \geq \left(1 - \left(1 - \dfrac{\beta p^{b}\big(1+ \frac{1}{p^{mb}}\big)}{p^{mb}}\right)^{t\sigma p^{mb}}\right).
\end{equation}  %

\end{proof}

Denote by $N(p,b)$ the maximum of $M(p,b)$ and $M^\prime(p,b)$.

\begin{proposition}
For any natural number $m$ that is greater than $N(p,b)$ and any positive real number $t$, there is a constant $C$ that is independent of $t$ so that for any $r$ in $(0,b)$, \[{\mathds E}_m[|Y_t|^r] \leq Ct^{\frac{r}{b}}.\]
\end{proposition}

\begin{proof}
For any positive real number $t$, there is an increasing sequence $(t_m)$ so that \[t_m\lambda_m = \floor{t\lambda_m}\quad \text{and} \quad \lim_{m\to \infty} t_m = t.\] The equality \[{\mathds E}_m[|Y_t|^r] = {\mathds E}_m^\ast[|S_{t_m\lambda_m}|^r]\] together with Theorem~\ref{Moment:Thm:Main}, with $n$ replaced by $t_m\lambda_m$, and Lemma~\ref{Convergence:lemma:ExtraTerm} implies that 
\begin{align*}
{\rm E}_m[|Y_t|^r] &\leq K_1(t_m\lambda_m)^{\frac{r}{b}}p^{-mr} + \left(1 - \left(1 - \dfrac{\beta p^{b}\big(1+ \frac{1}{p^{mb}}\big)}{p^{mb}}\right)^{t_m\sigma p^{mb}}\right)\frac{p^r(p-1)}{p^{r+1}-1}\\
& \leq K_1(t_m\sigma p^{mb})^{\frac{r}{b}}p^{-mr} + K_2(t_m\sigma)^{\frac{r}{b}}\\
&\leq Kt_m^{\frac{r}{b}}\sigma^{\frac{r}{b}} \leq Kt^{\frac{r}{b}}\sigma^{\frac{r}{b}},
\end{align*} %
since $t_m$ is bounded above by $t$.

\end{proof}

Henceforth, take $m$ to be greater than or equal to $N(p,b)$.

\begin{proposition}\label{tight}
Each stochastic process $(F([0,\infty)\colon \Zp), \Pb_m, Y)$ has a version with sample paths in $D([0,\infty)\colon \Zp)$, a process $(D([0,\infty)\colon \Zp), \Pb^m, Y)$. Furthermore, the sequence of measures $(\Pb^m)$ with paths in $D([0,\infty)\colon \Zp)$ is uniformly tight.
\end{proposition}

\begin{proof}
For any strictly increasing finite sequence $(t_1, t_2, t_3)$ in $[0, \infty)$, the independence of the increments of the process $(F(\mathds N_0\colon \Gm), \Pb^\ast_m, S)$ implies that 
\begin{align*}
\EXm\!\left[\big|Y_{t_3} - Y_{t_2}\big|^r\big|Y_{t_2} - Y_{t_1}\big|^r\right] & = {\mathds E}^\ast_m\!\left[\big|Y_{t_3} - Y_{t_2}\big|^r\right]{\mathds E}^\ast_m\!\left[\big|Y_{t_2} - Y_{t_1}\big|^r\right]\\ & \leq C(t_3-t_2)^\frac{r}{b}C(t_2-t_1)^\frac{r}{b} \leq C^2(t_3-t_1)^{\frac{2r}{b}}.
\end{align*}
Take $r$ to be any real number in $\left(\frac{b}{2}, b\right)$ to verify that $(F([0,\infty)\colon \Zp), \Pb_m, Y)$ satisfies the criterion of Chentsov, which implies that the stochastic process $(F([0,\infty)\colon \Zp), \Pb_m, Y)$ has a version with paths in $D([0,\infty)\colon \Zp)$, a process $(D([0,\infty)\colon \Zp), \Pbm, Y)$.  Uniformity of the constant $C$ in both $m$ and the triple of time points implies that $(\Pbm)$ is a uniformly tight sequence of probability measures \cite{cent}.  
\end{proof}

Continuity from above of the measure $\Pb^m$ together with the right continuity of the paths implies Proposition~\ref{prop:qp:concenration}.  The omitted proof is essentially identical to the proof of the same result in the earlier settings \cite{WJPA, W:Expo:24}.
 
\begin{proposition}\label{prop:qp:concenration}
The measure $\Pbm$ is concentrated on the subset of $\Gamma_m(\Gm)$-valued paths in $D([0,\infty)\colon \Zp)$ that are constant on each interval in $\big\{\big[(n-1)\tau_m, n\tau_m\big)\colon n\in\mathds N\big\}$.
\end{proposition}


\subsection{Weak convergence}

It is helpful to use more precise notation since there are two different types of quotient maps involved in this section.  Again denote by $\Gdmap{\cdot}$ the quotient map from $\mathds Q_p$ to $\mathds Q_p\slash \mathds Z_p$ and by $\Gdmap{\cdot}_m$ the quotient map from $p^{-m}\mathds Z_p$ to $\Gmdual$.  View $p^{-m}\mathds Z_p$ as a subset of $\mathds Q_p$ and extend functions from $p^{-m}\mathds Z_p$ to all of $\mathds Q_p$ by using cutoff functions.  Similarly, view $\Gmdual$ as a subset of $\Qp\slash\Zp$ and extend functions similarly.  For any positive real number $D$, take $\sigma$ to be the value \[\sigma = \frac{D}{\beta}.\]  This way of writing $\sigma$ gives a direct specification of the time scaling in order to have a diffusion constant $D$ for the limiting process.

Denote by $E_m$ the function that is defined for any $(t,\y)$ in $(0,\infty)\times \mathds Q_p\slash\Zp$ by %
\[
E_m(t, \y) = \begin{cases}\left(\phi_m(\y_m)\right)^{\floor{t\lambda(m)}} &\mbox{if }|\y| \leq p^{m}\\0 &\mbox{if }|\y| > p^{m}.\end{cases}%
\]

\begin{lemma}\label{last:Emlimit}
For any $t$ in $(0, \infty)$, \[\lim_{m\to \infty} \int_{\Qp\slash\Zp}\left|E_m(t,\y) - \e^{-D\left(|\y|^b-\beta^{-1}\right)t}\right|\,{\rm d}\y = 0.\]
\end{lemma}

\begin{proof}
For any $z$ in $[0, 1]$, the inequality \begin{equation*}0 \leq 1-z \leq \e^{-z}\end{equation*} implies that for any $m$ and any $x$ in $[0, m]$, \begin{equation}\label{ConofProcess:expineq:Dini}0 \leq \left(1-\frac{x}{m}\right)^m \leq \left(\e^{-\frac{x}{m}}\right)^m = \e^{-x}.\end{equation}  Since the sequence on the lefthand side of \eqref{ConofProcess:expineq:Dini} is increasing, for any positive real number $R$, Dini's theorem implies that $\left(1-\frac{x}{m}\right)^m$ converges uniformly to $\e^{-x}$ in $[0, R]$.  For any $\y$, if $|\y|$ is less than $p^m$ the Proposition~\ref{Prim:Prop:Boundonalphaoverpb} implies that \begin{equation}\label{COP:Lem:expext}\beta|\y|^b \leq p^{mb},\end{equation} and so \eqref{ConofProcess:expineq:Dini} together with the integrability of the exponential function over $(-\infty, 0]$ implies that there is a divergent, increasing, positive sequence $(R_m)$ in $(0, p^{m-1}]$ so that \begin{equation}\label{COP:Lem:expextSeq}\lim_{m\to \infty} \int_{|\y|\leq R_m}\left|E_m(t,\y) - \e^{-D\left(|\y|^b-\beta^{-1}\right)t}\right|\,{\rm d}\y = 0.\end{equation}  Denote by $A_m$ the set \[A_m = \{\y \in \Qp\slash\Zp\colon R_m<|\y|<p^m\; \text{or} \; |\y| > p^m\}.\]  The inequality \eqref{COP:Lem:expext} implies that for any $\y$ in $A_m$, \[\left|E_m(t,\y) - \e^{-D\left(|\y|^b-\beta^{-1}\right)t}\right| \leq \e^{-D\left(|\y|^b-\beta^{-1}\right)t},\] and so \begin{equation}\label{COP:Lem:expextFar}\lim_{m\to \infty} \int_{A_m}\left|E_m(t,\y) - \e^{-D\left(|\y|^b-\beta^{-1}\right)t}\right|\,{\rm d}\y = 0.\end{equation}  Since $|1- \beta + p^{-mb}|$ is in $[0,1)$ if $m$ is large enough, \begin{align}\label{COP:Lem:expextSingle}&\lim_{m\to \infty} \int_{|\y|=p^m}\left|E_m(t,\y) - \e^{-D\left(|\y|^b-\beta^{-1}\right)t}\right|\,{\rm d}\y \notag\\&\hspace{1.5in}\leq \lim_{m\to \infty} p^{md}\left|(1 - \beta - p^{-mb})^{\floor{t\lambda_m}} - \e^{-D\left(|\y|^b-\beta^{-1}\right)t}\right| = 0.\end{align}  Decompose the integral in the statement of the lemma into three regions, the ball of radius less than $R_m$, the set $A_m$, and the circle of radius $p^m$ and use \eqref{COP:Lem:expextSeq}, \eqref{COP:Lem:expextFar}, and \eqref{COP:Lem:expextSingle} to obtain the desired limit. 
\end{proof}

The set $H_R$ of \emph{restricted histories} for paths in $D([0, \infty)\colon \Zp)$ is the set of all histories whose route is a finite sequence of balls.  

\begin{proposition}\label{5:prop:restrictedhistconv}
For any history $h$ in $H_R$ whose route is a finite sequence of balls, \[\Pb^m(\C(h)) \to \Pb(\C(h)).\]
\end{proposition}

\begin{proof}
For any ball $B$ of radius $p^{-m}$ in $\Zp$, if $x$ is in $B$, then $\x_m$ is the unique element of $\Gm$ so that $\Gamma_m(\x_m)$ is in $B$, and so \eqref{Convergence:Eq:neMeas} implies that%
\begin{align*}
\Pbm(Y_t\in B) &= \Pb_m\big(S_{t_m\lambda_m}\in \Gamma_m^{-1}(B)\big).
\end{align*}
Denote by $\rho^m$ the function that for each $t$ in $(0,\infty)$ is given by %
\begin{equation}
\rho^m(t, z) = \sum_{\x \in \Gm} \rho\big(t_m\lambda_m,\x\big)\mathds 1_{x + p^m\Zp}(z), %
\end{equation}
so that for any $t$ in $(0,\infty)$ and any set $U$ in $\Zp$ that is a disjoint union of balls that each have radius at least $p^{-m}$, %
\[
\Pb^m(Y_t\in U) = \int_{\Zp}\rho^m(t,x)\,{\rm d}x = \Pb_m^\ast(S_{t_m\lambda_m}\in\Gmap{U}).
\]%

Take $\rho$ to be the probability mass function for the $\Zp$-valued process.  The absolute value of the difference between $\rho^m$ and $\rho$ is given by %
\begin{align}\label{5:prop:restrictedhistconv:unlim}
\left|\rho^m(t, x) - \rho(t, x)\right| & \leq \int_{\Qp\slash\Zp} \left|\chi(xy) E_m(t,\y) - \chi(xy)\e^{-D\left(|\y|^b-\beta^{-1}\right)t}\right|\,{\rm d}\y\notag\\& = \int_{\Qp\slash\Zp} \left|E_m(t,\y) - \e^{-D\left(|\y|^b-\beta^{-1}\right)t}\right|\,{\rm d}\y = \varepsilon_m(t) \to 0.
\end{align} %
The sequence $(\varepsilon_m(t))$ is independent of $\x$, which implies that 
$(\rho_m(t, \cdot))$ converges uniformly on $\Zp$ to $\rho(t, \cdot)$.

Take $h$ to be any history whose route is a finite sequence of balls and without loss in generality suppose that $U_h(0)$ is the set $\{0\}$.  Simplify the notation by writing \[e_h = (t_0, \dots, t_k) \quad \text{and}\quad U_h = (\{0\}, U_1, \dots, U_k).\]  For any $i$ in $\{0, \dots, k\}$, denote by $r_i$ the radius of $U_i$. For any $m$ so that \[p^{-m} < \min\{r_1, \dots, r_k\},\] the uniform convergence given by \eqref{5:prop:restrictedhistconv:unlim} implies that %
\begin{align*}%
\Pbm(\C(h)) &= \Pbm(Y_{t_1}\in U_1, Y_{t_2}\in U_2, \dots, Y_{t_n}\in U_n)\\
&= \int_{U_1} \cdots \int_{U_k} \prod_{i\in\{1, \dots, k\}}\rho^m(t_i-t_{i-1}, x_i -x_{i-1})\,{\rm d}\mu(x_k)\cdots {\rm d}\mu(x_1)\\
&\to  \int_{U_1}\cdots \int_{U_k} \prod_{i\in\{1, \dots, k\}}\rho(t_i-t_{i-1}, x_i-x_{i-1})\,{\rm d}\mu(x_k)\cdots {\rm d}\mu(x_1) = \Pb(\C(h)).
\end{align*}
\end{proof}

\begin{proof}[Proof of Theorem~\ref{Sec:Con:Theorem:MAIN}]
The intersection of any two balls in $\Zp$ is again a ball in $\Zp$, and so $\C(H_{R})$ is a $\pi$-system that generates the $\sigma$-algebra of cylinder sets of paths in $D([0, \infty)\colon \Zp)$.  The uniform tightness of the family of measures $\{\Pbm\colon m\in \mathds N_0\}$ that Proposition~\ref{tight} guarantees together with the convergence for any restricted history $h$ of $(\Pbm(\C(h)))$ to $\Pb(\C(h))$ implies Theorem~\ref{Sec:Con:Theorem:MAIN}.
\end{proof}

As in the real case and the $p$-adic case \cite{BW, WJPA}, the scaling factor for the time scales is proportionate to a power of the scaling factor for the space scales.  The power is the exponent of the Vladimirov-Kochubei operator, and is independent of $p$.


\begin{thebibliography}{0}

%
\bibitem{ABKO:JPA:2002} Avetisov, V.A., Bikulov, A.H., Kozyrev, S.V., Osipov, V.A.: \textsl{$P$-Adic models of ultrametric diffusion constrained by hierarchical energy landscapes}. J. Phys. A: Math. Gen. 35 (2), 177, (2002).

%
\bibitem{Bik:UAA:2010} Bikulov, A.K.: \textsl{Problem of the first passage time for $p$-adic diffusion}. $P$-Adic Numbers, Ultrametric Analysis, and Applications, 2(2), 89-99, (2010).

%
\bibitem{Avetisov_Bikulov_Kozyrev:JPA:1999}  Avetisov, V.A., Bikulov, A.H., Kozyrev, S.V.: \textsl{Application of $p$-adic analysis to models of spontaneous breaking of replica symmetry.} J. Phys. A: Math. Gen. 32 (50), 8785--8791, (1999).

%
\bibitem{ABZ:ProcSteklov:2014} Avetisov, V.A., Bikulov, A.K., Zubarev, A.P.: \textsl{Ultrametric random walk and dynamics of protein molecules}. Proc. Steklov Inst. Math. 285, 3-25, (2014). 

%
\bibitem{BW} Bakken, E., Weisbart, D.:  \textsl{$p$\,-Adic brownian motion as a limit of discrete time random walks}. Commun. Math. Phys. 369, 371-402, (2019).

%
\bibitem{BV:IzvMath:1997} Bikulov, A.K., Volovich, I.V.: \textsl{$p$-Adic Brownian motion}. Izvestiya: Mathematics, 61 (3), 537, (1997).

%
\bibitem{BonoforteVazquez:NLA:2016} Bonforte, M., V\'{a}zquez, J.L.: \textsl{Fractional nonlinear degenerate diffusion equations on bounded domains}. Nonlin. Anal., 131, 363-398, (2016).

%
\bibitem{BZ:Physica:2021} Bikulov, A.K., Zubarev, A.P.: \textsl{Ultrametric theory of conformational dynamics of protein molecules in a functional state and the description of experiments on the kinetics of CO binding to myoglobin}. Physica A: Statistical Mechanics and its Applications, 583, 126280, (2021).

%
\bibitem{bil1} Billingsley, P.: \textsl{Convergence of Probability Measures, Second Edition}. John Wiley $\&$ Sons, (1999).

%
\bibitem{cent} Chentsov, N.N.: \textsl{Weak convergence of stochastic processes whose trajectories have no discontinuities of the second kind and the ``heuristic'' approach to the Kolmogorov--Smirnov tests}. Theory of Probability \& Its Applications, 1(1):140-144, (1956).

%
\bibitem{Khrennikov:JFAA:2016} Khrennikov, A., Oleschko, K., Correa Lopez, M.J.:  \textsl{Application of $p$-adic wavelets to model reaction-diffusion dynamics in random porous media}. J. Fourier Anal. Appl., 22, 809-822, (2016).

%
\bibitem{Khrennikov:Entropy:2016} Khrennikov, A., Oleschko, K., Correa Lopez, M.J.: \textsl{Modeling fluid’s dynamics with master equations in ultrametric spaces representing the treelike structure of capillary networks}. Entropy, 18, Art. 249, (2016).

%
\bibitem{Koch:UMJ:2018} Kochubei, A.N.: \textsl{Linear and Nonlinear Heat Equations on a $p$-Adic Ball}. Ukr Math J 70, 217- 231, (2018).

%
\bibitem{Kochubei:Book:2001} Kochubei, A.N.: \textsl{Pseudo-Differential Equations and Stochastics over non-Archimedean Fields.} Monographs and Textbooks in Pure and Applied Mathematics 244 (Marcel Dekker Inc., New York, 2001).

%
\bibitem{Parisi:PRL:1979} Parisi, G.: \textsl{Infinite number of order parameters for spin-glasses.} Phys. Rev. Lett. 43, 1754, (1979).

%
\bibitem{Parisi:JPA:1980} Parisi, G.: \textsl{A sequence of approximate solutions to the S-K model for spin glasses.} J. Phys. A: Math. Gen. 13, (1980)

%
\bibitem{Parisi_Sourlas:JPA:1999} Parisi, G., Sourlas, N.: \textsl{$p$-Adic numbers and replica symmetry breaking.} Europ. Phys. J. B 14, 535-542, (2000)

%
\bibitem{W:Expo:24} Pierce, T., Rajkumar, R., Stine, A., Weisbart, D., Yassine, A.M.: \textsl{Brownian Motion in a Vector Space over a Local Field is a Scaling Limit}. arXiv:2405.02502

%
\bibitem{RW:JFAA:2023} Rajkumar, R., Weisbart, D.: \textsl{Components and exit times of Brownian motion in two or more $p$-adic dimensions}. J. Fourier Anal. Appl. 29 (6), 75, (2023).

%
\bibitem{Varadarajan:LMP:1997} Varadarajan, V.S.: \textsl{Path integrals for a class of $p$-adic Schr{\"o}dinger equations}. Lett. Math. Phys. 39, no. 2, 97-106, (1997).

%
\bibitem{Vlad90} Vladimirov, V.S.: \textsl{On the spectrum of some pseudo-differential operators over $p$-adic number field}. Algebra and analysis 2, 107-124, (1990).

%
\bibitem{Weisbart:2021} Weisbart, D.: \textsl{Estimates of certain exit probabilities for $p$-adic Brownian bridges}.  J. Theor. Probab., (2021).

%
\bibitem{WJPA} Weisbart, D.: \emph{$p$-Adic Brownian Motion is a Scaling Limit}. Journal of Physics A: Mathematical and Theoretical, (2024). 10.1088/1751-8121/ad40df.


\end{thebibliography}
\end{document}